\documentclass[11pt,letterpaper]{svmult}
\usepackage{blindtext}
\usepackage{amsthm}
\usepackage{amsmath}
\usepackage{amsfonts}
\usepackage{amssymb}
\usepackage{mathrsfs}
\usepackage{hyperref}
\usepackage{cleveref}
\newtheorem{thm}{Theorem}[section]
\newtheorem{lem}[thm]{Lemma}
\newtheorem{defn}[thm]{Definition}
\newtheorem{prop}[thm]{Proposition}
\newtheorem{cor}[thm]{Corollary}
\theoremstyle{remark}
\newtheorem*{thm*}{Theorem}

\newtheorem*{conj}{Conjecture}

\counterwithin{equation}{section}

\newtheorem{rem}[thm]{Remark}

\newtheorem{fact}[thm]{Fact}

\newcommand{\bbN}{\mathbb{N}}

\begin{document}
\title*{The geometry of branched coverings of hyperbolic manifolds}
\titlerunning{The geometry of branched covers}
\author{Ursula Hamenst\"adt}
\institute{Mathematisches Institut der Universität Bonn, Endenicher Allee 60, 53115 Bonn, Germany, \email{ursula@math.uni-bonn.de} 
Partially supported by the DFG Schwerpunktprogramm "Geometry at infinity"
}
\date{May 11, 2024}

\maketitle

\abstract{
We discuss geometric properties of 
covers of closed hyperbolic manifolds of dimension $n\geq 3$, branched along 
a totally geodesic codimension two submanifold $\Sigma$.
We survey what is known about the existence of Einstein metrics on such manifolds. 
In dimension $n\geq 4$, we show that
for suitable branch loci 
$\Sigma$, at most one of these branched coverings admits a
hyperbolic metric. 
 }

\section{Introduction}

As a consequence of the solution to the geometrization conjecture by Perelman, 
any closed manifold of dimension three which admits a negatively curved metric
also admits a hyperbolic metric, that is, a metric of constant curvature $-1$. 
The existence of hyperbolic metrics on closed surfaces of negative
Euler characteristic is 
a classical consequence of the uniformization theorem. An analogous property is not 
true any more for closed negatively curved manifolds of dimension at least four.

Gromov and Thurston \cite{GT87}  considered finite
cyclic coverings of arithmetic hyperbolic manifolds of simplest type and dimension
$n\geq 4$, branched along a 
null-homolo\-gous totally 
geodesic submanifold of codimension two. Such manifolds will be called
\textit{Gromov Thurston manifolds} in the sequel.
They showed that suitable choices of 
such manifolds
admit metrics whose sectional curvature is contained in the interval $[-1-\epsilon,-1]$
where $\epsilon >0$ can be arbitrarily prescribed, 
but which 
do not admit a hyperbolic metric. 
The proof of non-existence of a hyperbolic
metric on such manifolds is however indirect, that is, 
it it shown that among an infinite collection of candidate manifolds with
pinched curvature, at most  finitely many 
admit hyperbolic metrics.

This leads to the question of the existence of distinguished 
metrics on Gromov Thurston manifolds. A class of metrics which generalizes constant curvature 
metrics on 3-manifolds are \emph{Einstein metrics} $h$, characterized by the property that
their Ricci curvature ${\rm Ric}_h$ satisfies ${\rm Ric}_h= ch$ for a constant $c\in \mathbb{R}$.

Part of the following conjecture was phrased as a question in \cite{GT87}.

\begin{conj}
  A nontrivial finite covering of a closed hyperbolic manifold $M$ of dimension $n\geq 4$, branched
  along a closed totally geodesic submanifold of codimension two, admits a unique Einstein metric up to scaling,
  and it does not admit a metric of constant curvature.
  \end{conj}

Progress towards this conjecture 
is due to Fine and Premoselli \cite{FP20} who constructed
negatively curved Einstein metrics of non-constant curvature 
on some four-dimensional
Gromov Thurston manifolds.
That in dimension four closed hyperbolic manifolds admit a unique Einstein
metric up to scaling and isometry
follows from
the work of Besson, Courtois and Gallot (see Corollary 4.6 of \cite{And10} for an explicit statement). 
Thus the last part of the above conjecture holds for some four-dimensional Gromov Thurston manifolds. 
The results in \cite{FP20} were extended in 
\cite{HJ24} as follows. 

\begin{theorem}[Hamenst\"adt-J\"ackel \cite{HJ24}]\label{main1}
For any $n\geq 4$, there are Gromov Thurston manifolds 
of dimension $n$ which admit a negatively curved 
Einstein metric but no hyperbolic metric. 
\end{theorem}

The article \cite{GH25} contains an analogous result in the K\"ahler case.

The fact that the examples of \cite{HJ24} 
do not admit a hyperbolic metric uses some additional property of the branched covers
considered. Namely, the branch locus $\Sigma$ of the base hyperbolic manifold, 
which is a closed totally geodesic submanifold of codimension two, is contained 
in a closed totally geodesic submanifold  
$H$ of codimension one, and it is homologous to zero in $H$. 
The manifold $H$ in turn is contained in the fixed point
set of an isometric involution of $M$. 
We call such a submanifold $\Sigma$ of a closed hyperbolic manifold $M$
\emph{special} in the sequel. 
In \cite{HJ24}, it was shown that among the covers of degree
a multiple of four, branched along a special submanifold, 
at most one can admit a hyperbolic metric.

The following result improves this statement using the arguments in \cite{HJ24}
and
collects some facts which are known to the experts but for which 
we could not find a reference in the literature.

\begin{theorem}\label{mainthm}
Let $M$ be an $n $-dimensional closed hyperbolic manifold and let $\Sigma\subset M$ be
a closed totally geodesic embedded oriented 
submanifold of codimension two. 
\begin{enumerate}
\item For a number $d\geq 2$ there exists a cyclic covering $M_d$ of $M$ of degree $d$, branched along 
$\Sigma$, if the class of $\Sigma$ in $H_{n-2}(M,\mathbb{Z}/d\mathbb{Z})$ vanishes.
\item If $n=3$ and if a cyclic branched covering $M_d$ of $M$ exists, then $M_d$ admits a hyperbolic metric, and  
the non-compact manifold $M\setminus \Sigma$ admits a complete 
hyperbolic metric of finite volume. Furthermore, $M$ can both be represented as a quotient of 
$\mathbb{H}^3$ by a cocompact torsion free lattice, and by a lattice which is not torsion free. 
\item If $n\geq 4$ and if $\Sigma$ is special, 
then $M_d$ admits a hyperbolic metric for at most one degree $d$.
\end{enumerate}
\end{theorem}

The organization of this article is as follows.
Section \ref{sec:branched} contains the proof of the first part of Theorem \ref{mainthm}. 
In Section \ref{sec:cone} 
we give a short  introduction to hyperbolic cone metrics and their relation 
to Theorem \ref{mainthm}. This is used in Section \ref{sec:branched3} to study branched covers
of hyperbolic 3-manifolds and review some results from \cite{HJ22}. 
The proof of Theorem \ref{mainthm} is contained in 
Section \ref{sec:branched4}. 

\bigskip

\noindent
{\bf Acknowledgement:} This article summarizes and  slightly 
extends results obtained within the framework 
of the Schwerpunktprogramm \emph{Geometry
at infinity}. I am grateful to Stephan Stadler for pointing the 
reference \cite{FS90} out to me.

\section{Branched covering}\label{sec:branched}

In this section we consider a
smooth closed oriented manifold $M$ of dimension $n\geq 3$ and 
a smooth closed oriented embedded
submanifold $\Sigma\subset M$ of codimension two.

\begin{defn}\label{cyclic} 
A \emph{cyclic $d$-fold covering of $M$ branched along $\Sigma$} 
is a manifold $M_d$ which admits a degree $d$ map $\Pi:M_d\to M$ with the following properties.
\begin{enumerate}
\item The restriction of $\Pi$ to $\Pi^{-1}(M\setminus \Sigma)$ is a degree $d$ regular covering onto 
$M\setminus \Sigma$, with deck group $\mathbb{Z}/d\mathbb{Z}$. 
\item $\Pi \vert \Pi^{-1}(\Sigma)$ is a homeomorphism. 
\end{enumerate}
\end{defn}

We refer to \cite{F57} for more information on more general branched coverings of simplicial complexes.

Let $\nu\to \Sigma$ be the normal bundle of $\Sigma$, which is an oriented two-dimensional 
vector bundle over $\Sigma$. 
The orientation
of $\nu$ can be used to equip $\nu$ with the structure of a complex line bundle. 
The complex line bundle $\nu\to \Sigma$  is defined by a classifying map 
$\Sigma\to \mathbb{C}P^\infty$, so that $\nu$ is the pull-back of the
tautological bundle $\tau\to \mathbb{C}P^\infty$ under this map.
The Euler class $e(\nu)$ of $\nu$ then is the pull-back of the generator $e(\tau)$ of 
$H^2(\mathbb{C}P^\infty;\mathbb{Z})$ under the classifying map. 

Since the odd dimensional homology of $\mathbb{C}P^\infty$ vanishes, the universal coefficient
theorem shows that any class $a\in H^2(\mathbb{C}P^\infty;\mathbb{Z})$ can be identified with a 
homomorphism $H_2(\mathbb{C}P^\infty;\mathbb{Z})\to \mathbb{Z}$. For every $d\geq 2$, 
the mod $d$ reduction $e(\tau)_d$ 
of $e(\tau)$ is the class in $H^2(\mathbb{C}P^\infty;\mathbb{Z}/d\mathbb{Z})$ which 
arises from composing the homomorphism defining $e(\tau)$ with the quotient map 
$\mathbb{Z}\to \mathbb{Z}/d\mathbb{Z}$. The mod $d$ reduction of the Euler class of $\nu$ then is the
pull-back $e(\nu)_d \in H^2(\Sigma; \mathbb{Z}/d\mathbb{Z})$ 
of $e(\tau)_d$ 
by the classifying map.

For every integer $d\geq 2$, the mod $d$ fundamental class of 
$\Sigma$ is the generator of the group 
$H_{n-2}(\Sigma;\mathbb{Z}/d\mathbb{Z})$ determined by the choice of an orientation.
Let $[\Sigma]_d$ be its image under the map $H_{n-2}(\Sigma;\mathbb{Z}/d\mathbb{Z})\to 
H_{n-2}(M;\mathbb{Z}/d\mathbb{Z})$
induced by the inclusion $\Sigma\to M$.

A cyclic covering of $M$ branched 
along $\Sigma$ restricts to a cyclic branched covering
$\Pi\vert \Pi^{-1}(U)\to U$ 
of a tubular neighborhood $U$ of $\Sigma$ in $M$, and $U$ can be chosen to be diffeomorphic
to the total space of the normal bundle $\nu$. With this identification and up to homotopy,
this cyclic covering restricts
to a cyclic unbranched covering of the complement of zero
in each fiber of $\nu$. 

\begin{lem}\label{euler}
There is a $d$-fold cyclic covering of $U$ branched along $\Sigma$ if and only if 
 the mod $d$ reduction of 
the Euler class $e(\nu)_d\in H^2(\Sigma,\mathbb{Z}/d\mathbb{Z})$ vanishes. 
\end{lem}
\begin{proof} Let $\nu_0\subset \nu$ be the complement of the zero section of $\nu$.
  Let $R$ be a ring with unit and 
  consider the exact cohomology sequence of the pair $(\nu,\nu_0)$ given by 
  \begin{equation}\label{exact}
    \cdots \to H^1(\nu_0;R)\to H^2(\nu,\nu_0;R)\to H^2(\nu;R)\to
  H^2(\nu_0;R)\to \cdots\end{equation} 
  
  For $x\in \Sigma$ let $\nu_x$ be the oriented fiber of $\nu$ over $x$.
  The \emph{Thom class} $\tau(\nu)$ of $\nu$ is a distinguished element of
  $H^2(\nu,\nu_0;\mathbb{Z})$ which for each $x\in \Sigma$ restricts to the preferred generator of
  $H^2(\nu_x,\nu_x\setminus 0;\mathbb{Z})$
  (Theorem 10.4 of \cite{MS74}). Using the Gysin sequence Theorem 12.2 of \cite{MS74} or 
  Proposition 6.41 of \cite{BT82}, the image of $\tau(\nu)$ under the restriction map
  $H^2(\nu;\nu_0;\mathbb{Z})\to H^2(\nu;\mathbb{Z})$
  is the 
  Euler class $e(\nu)$ of $\nu$ via the homotopy equivalence of the
  total space of $\nu$ with $\Sigma$.

  Now the fundamental group of a regular $d$-fold
  cyclic covering of $\nu_0$ restricting to a $d$-fold covering of a fiber is
  the kernel of a homomorphism $\pi_1(\nu_0)\to \mathbb{Z}/d\mathbb{Z}$ which restricts to
  a surjective homomorphism of the fundamental group $\mathbb{Z}$  of the fiber of
  $\nu_0$ onto $\mathbb{Z}/d\mathbb{Z}$.
  Such a homomorphism, viewed as a nontrivial element in $H^1(\nu_0;\mathbb{Z}/d\mathbb{Z})$, maps to the 
  mod $d$ reduction $\tau_d(\nu)\in H^2(\nu,\nu_0; \mathbb{Z}/d\mathbb{Z})$ of the Thom class 
  in the exact cohomology sequence (\ref{exact}) for $R=\mathbb{Z}/d\mathbb{Z}$. 
    By exactness, 
  such a homomorphism exists 
  if and only if the mod $d$ reduction of the Euler
  class $e_d(\nu)$ of $\nu$ vanishes.
\end{proof}

For a ring $R$ with unit consider now the exact sequence
\[\cdots \to H^1(M\setminus \Sigma;R)\to H^2(M,M\setminus \Sigma;R)\to H^2(M,R)\to
  H^2(M\setminus \Sigma;R)\to \cdots\]
By excision, the Thom class of the normal bundle $\nu$ defines a nontrivial class
in $H^2(M,M\setminus \Sigma;R)=H^2(U,U\setminus \Sigma;R)=
H^2(\nu,\nu_0;R)$. 
By Proposition 6.24 of \cite{BT82},
its image under the map
$H^2(M,M\setminus \Sigma;R)\to H^2(M;R)$ is just the Poincar\'e dual of $\Sigma$ 
 with coefficients in $R$. 

Taking  $R=\mathbb{Z}/d\mathbb{Z}$, this
class vanishes if and only if the mod $d$ reduction $\tau(\nu)_d$ of the Thom class of $\nu$ is contained in the
image of $H^1(M\setminus \Sigma;\mathbb{Z}/d\mathbb{Z})$. Then
the kernel of the homomorphism $\pi_1(M\setminus \Sigma)\to \mathbb{Z}/d\mathbb{Z}$ defined 
by an element of $H^1(M\setminus \Sigma;\mathbb{Z}/d\mathbb{Z})$ which maps to $\tau(\nu)_d$ 
is an unbranched degree $d$ 
covering of $M\setminus \Sigma$.  
By excision and Lemma \ref{euler}, this covering extends to 
a cyclic
covering of $M$ branched along $\Sigma$. 
Together with Poincar\'e duality, this shows 
the first part of Theorem \ref{mainthm}.

\begin{prop}\label{prop:branched}
There exists a cyclic $d$-fold covering of $M$ branched along $\Sigma$ if 
$[\Sigma]_d=0\in H_{n-2}(M,\mathbb{Z}/d\mathbb{Z})$. 
\end{prop}

The following example shows that the condition in Proposition \ref{prop:branched} for the existence
of a branched covering is not necessary. 

\begin{example}
Let $S_1,S_2$ be two closed oriented surfaces of genus $g\geq 2$. Then for any $p\in S_2$, 
$\Sigma=S_1\times \{p\}$ is a codimension 2 oriented embedded submanifold of $S_1\times S_2$ with
$[\Sigma]_d\not=0$ for any $d\geq 1$. On the other hand, coverings of $S_1\times S_2$ branched along
$\Sigma$ correspond precisely to coverings of $S_2$ branched at $p$, and these exist
for any odd degree $d\geq 3$ \cite{PP06}. 
\end{example}

\section{Branched coverings and hyperbolic cone metrics}\label{sec:cone}

Let $\mathbb{H}^n$ be the hyperbolic space and let $\mathbb{H}^{n-2}$ be a 
totally geodesic subspace of codimension two. The subgroup of the group
${\rm PO}(n,1)$ of orientation preserving isometries of 
$\mathbb{H}^n$ which fixes $\mathbb{H}^{n-2}$ pointwise is the circle 
group $S^1$ acting on a fiber of the normal bundle as a group of rotations. 
The fundamental group of $\mathbb{H}^n\setminus \mathbb{H}^{n-2}$ equals the group 
$\mathbb{Z}$. 

If we denote by $\tilde X$ the universal covering of $\mathbb{H}^n\setminus \mathbb{H}^{n-2}$,
then the abelian group $\mathbb{R}$, which is the universal covering of the circle $S^1$, 
acts freely and isometrically on $\tilde X$ with respect to the
(incomplete) hyperbolic pull-back metric. Thus we can take the quotient $X$ of $\tilde X$ under 
an infinite cyclic  subgroup of $\mathbb{R}$. The metric completion $\bar X$ of $X$ 
contains the isometrically embedded subspace $\mathbb{H}^{n-2}$, and a fiber of the normal bundle 
has a natural identification with a two-dimensional
hyperbolic cone with angle $\alpha\in (0,\infty)$.  We call $\mathbb{H}^{n-2}$ the \emph{singular set}
of the hyperbolic cone metric. 

If $\alpha\in (0,2\pi]$
then a fiber cone   
is obtained as follows. In the disk model $D=\{z\mid \vert z\vert <1\}$ for the hyperbolic plane, straight line segments starting 
at the origin are geodesics up to parameterization. Cut $D$ open along the segments
$\rho=\{\Im =0, \Re \geq 0\}$ and 
$e^{i\alpha}\rho$. 
One component of the resulting space is a sector of angle $\alpha$ 
with geodesic boundary. 
Glue this component 
along
the boundary with the rotation $e^{i\alpha}$. The hyperbolic metric on $D$ descends to a hyperbolic cone metric
with cone angle $\alpha$ at the image of the vertex $0$ of the sector. 

We call a metric on a closed 
manifold which is hyperbolic outside of a closed codimension two submanifold $\Sigma$ 
and such that any point $x\in \Sigma$ has a neighborhood which is isometric 
to a neighborhood of a point in the singular set of a hyperbolic cone metric
as described above
a \emph{hyperbolic cone metric} as well. 
The cone angle is locally constant.  

Assume now that $M$ is a 
closed oriented hyperbolic manifold 
containing a closed totally geodesic oriented submanifold $\Sigma$ of codimension two 
which is homologous to zero. 
By Proposition \ref{prop:branched}, 
for each $d\geq 2$ we then can construct a $d$-fold cyclic cover 
$\Pi:M_d\to  M$ branched along 
$\Sigma$. The pull-back under $\Pi$ of the hyperbolic metric on $M$ is a hyperbolic cone metric
with cone angle $2d\pi$ and singular set the preimage of $\Sigma$ under the natural covering projection.
This observation is summarized as follows.

\begin{lem}\label{conemetric}
The $d$-fold cyclic covering $M_d$ of $M$ branched along $\Sigma$ admits a hyperbolic cone metric with 
conical singularity along $\Sigma$ and cone angle $2\pi d$.
\end{lem}
 
For a hyperbolic cone metric, the volume form is defined and gives rise to a volume by integration.
The following is immediate from the above construction.

\begin{cor}\label{cor:cone}
The volume of $M_d$ with respect to the hyperbolic cone metric equals $d{\rm vol}(M)$. 
\end{cor}
\begin{proof}
As the branched covering projection $\Pi:M_d\to M$ is a map of degree $d$ which is furthermore 
an orientation preserving local diffeomorphism on $M_d\setminus \Sigma$, the volume form of the 
cone metric equals the pull-back of the volume form of $M$ outside of a submanifold of codimension two, 
and this form integrates to 
$d{\rm vol}(M)$.
\end{proof}

We next discuss another viewpoint of the hyperbolic cone metric $g_d$ on $M_d$. 
Namely, consider again the hyperbolic space $\mathbb{H}^n$ and a totally geodesic 
subspace $\mathbb{H}^{n-2}$ of codimension two.
The stabilizer 
${\rm Stab}(\mathbb{H}^{n-2})$ of $\mathbb{H}^n$ in 
${\rm PO}(n,1)$ which consists of isometries whose restrictions to
$\mathbb{H}^{n-2}$ are orientation preserving can be
  identified with
  $S^1\times {\rm PO}(n-2,1)$
where $S^1$ is the subgroup fixing $\mathbb{H}^{n-2}$ pointwise.

The hyperbolic cone metric on $M_d$ pulls back to a hyperbolic cone metric on the covering 
$\hat M_d$ of 
$M_d$ with fundamental group $\pi_1(\Sigma)$. The singular locus of this metric is the 
submanifold $\Sigma$. The cyclic deck group $\Gamma$ of $M_d$ acts on 
the cone manifold $\hat M_d$ as a cyclic group 
of isometries fixing $\Sigma$ pointwise. If $h$ is a smooth Riemannian metric on 
$M_d$ which is invariant under the action of $\Gamma$, then $h$ lifts to a smooth metric $\hat h$ 
on $\hat M_d$ 
which is invariant under $\Gamma$. Consequently $h$ descends to a cone metric on 
$\Gamma\backslash \hat M_d=\hat M$ with singular locus $\Sigma$ and cone angle $\frac{2\pi}{d}$.
This manifold in turn is just the quotient of $\mathbb{H}^n$ under the action of the fundamental group of 
$\Sigma$.

Thus if one is interested in understanding distinguished metrics on $M_d$ one is led to studying 
distinguished cone metrics on $\mathbb{H}^n$ with singular locus $\mathbb{H}^{n-2}$ and cone angle 
$\frac{2\pi}{d}$ which are invariant under the action of the group 
$S^1\times {\rm PO}(n-2,1)$. 
Particular such metrics 
are invariant Einstein 
cone metrics on $\mathbb{H}^{n}$ with singular locus $\mathbb{H}^{n-2}$ 
and cone angle $\frac{2\pi}{d}$. 
In order to be the lift of a metric on $M_d$, it is necessary that the metric 
glues to a metric on $M_d\setminus \Sigma$. This motivates the study of such Einstein cone metrics 
on $\mathbb{H}^n$ 
which are asymptotic to the hyperbolic metric as the distance from $\mathbb{H}^{n-2}$ tends to infinity. 
In particular, the Einstein constant has to be negative and hence by rescaling, one may assume that
the Einstein constant equals $-(n-1)$.

For three-dimensional manifolds, the Einstein condition is equivalent to constant curvature \cite{Bes08}, but this
is not true any more in higher dimension. 
The following is due to Fine-Premoselli (Proposition 3.2 of \cite{FP20}). 
An analogous statement in the K\"ahler setting 
was established in \cite{GH25}.

\begin{prop}[Proposition 3.2 and Lemma 3.3 of \cite{FP20}]\label{prop:cone}
Let us con\-si\-der a to\-tally geodesic subspace  $\mathbb{H}^{n-2}\subset \mathbb{H}^n$ 
and let $\alpha \in (0,\infty)$. Then there exists 
a $S^1\times {\rm PO}(n-2,1)$-invariant  
Einstein cone metric $g_\alpha$ on $\mathbb{H}^n$ with Einstein constant $-(n-1)$ and 
the following properties.
\begin{enumerate}
\item The singular locus of $g_\alpha$ equals $\mathbb{H}^{n-2}$, 
and $\mathbb{H}^{n-2}$ is totally geodesic for $g_\alpha$. 
\item The cone angle of $g_\alpha$ along $\mathbb{H}^{n-2}$ equals $\alpha$.  
\item The sectional curvature of $g_\alpha$ is negative. 
\item As $d(\mathbb{H}^{n-2},x)\to \infty$, the metric $g_\alpha$ converges smoothly to the hyperbolic metric $g$.
\end{enumerate}
We have $\alpha < 2\pi$ if and only if the restriction of $g_\alpha$ to $\mathbb{H}^{n-2}$ satisfies 
$g_\alpha\vert \mathbb{H}^{n-2}< g\vert \mathbb{H}^{n-2}$. In particular, 
for $\alpha\not=2\pi$, the metric $g_\alpha$ is a hyperbolic cone metric if and only if $n=3$. 
\end{prop}
%
%
%

Proposition \ref{prop:cone} allows to glue the restriction of the 
singular Einstein metric $g_\alpha$ to a collar neighborhood of $\Sigma$
in $\pi_1(\Sigma)\backslash \mathbb{H}^n$ which embeds into $M$ 
to the hyperbolic metric on $M\setminus \Sigma$ and construct in this way a cone metric on $M$ whose
pull-back to $M_d$ is  
a smooth metric on $M_d$ which is 
Einstein on the complement
of the gluing region. Provided that the collar neighborhood is large, explicit control on the gluing 
construction results in geometric control of the glued metric. 
Theorem \ref{main1} is proved by deforming such a well enough controlled 
metric to an Einstein metric using an 
implicit function theorem. This strategy is also used in \cite{GH25} in the K\"ahler case.
In \cite{J25}, building on the earlier work \cite{HJ22}, 
a different but related approach is used  to show that in any dimension $n\geq 4$, one can 
deform suitably pinched negatively curved 
metrics which are close to Einstein metrics in a suitable sense to Einstein metrics.

\section{Branched coverings of hyperbolic 3-manifolds}\label{sec:branched3}

Consider now a closed hyperbolic 3-manifold $M$ and an embedded oriented 
geodesic multicurve $c\subset M$ which 
is homologous to zero. By
Proposition \ref{prop:branched}, we know that for every 
$d\geq 2$ there exists a covering $M_d\to M$ of degree $d$, branched along $c$. 
It follows from the discussion in Section \ref{sec:cone} that this branched covering manifold
admits a hyperbolic cone metric with cone angle $2\pi d>2\pi$ which can be
deformed to a 
smooth negatively curved metric. We refer to \cite{GT87} (see also \cite{FS90}) for an explicit construction
of such a metric which can be chosen to be a warped product metric near the singular set of the cone metric.
Thus as a consequence of hyperbolization, there also exists a hyperbolic metric on $M_d$. 
Note that this is consistent with the fact that 
the Einstein cone metric on $\mathbb{H}^n$ with singular locus $\mathbb{H}^{n-2}$ whose existence was stated in 
Proposition \ref{prop:cone} is hyperbolic if and only if $n=3$.

Since $M_d$ is a cyclic branched cover of $M$, it admits a deck group action by a cyclic group of 
order $d$ of diffeomorphisms, with quotient 
$M$. The fixed point set of a generator 
$\zeta$ of this group 
is the preimage of the geodesic multicurve $c$ under the covering map.
By Mostow rigidity, there exists a cyclic group of order $d$ of isometries of $M_d$
for the hyperbolic metric.
This means that there exists an isometry $\zeta^\sharp$ of $M_d$ which is homotopic to $\zeta$ and 
such that $(\zeta^\sharp)^d={\rm Id}$.

As follows from the work \cite{CLW18}, in general it is difficult to determine the relation between the fixed point 
sets of $\zeta$ and $\zeta^\sharp$. However, in the specific situation of a cyclic branched covering, 
Proposition 5.8 of \cite{HJ24} and the following remark shows that 
the fixed point set 
${\rm Fix}(\zeta^\#)\subset X$ of 
$\zeta^\#$ 
is (abstractly) diffeomorphic to ${\rm Fix}(\zeta)$. 
Moreover, ${\rm Fix}(\zeta^\#)$ and ${\rm Fix}(\zeta)$ are freely homotopic inside $X$.
As a consequence, the quotient of $M_d$ by the group $\langle \zeta^\sharp\rangle$
generated by $\zeta^\sharp$ is a hyperbolic 
orbifold, or a
cone manifold, which is diffeomorphic to $M$, with singular locus homotopic to 
$c$. The cone angle equals 
$2\pi/d$. 

\begin{question}
Is it true that the length of $c$ for the smooth hyperbolic metric on 
$M_d$ is strictly smaller than the length of $c$ on $M$?
\end{question}

Note that as $M$ equipped with the cone metric is a quotient of $M_d$ equipped with a hyperbolic 
metric by a cyclic group of isometries, this shows that $M$ admits both the structure of a hyperbolic
manifold as well as a different structure of a hyperbolic orbifold.  
Since the quotient $\langle \zeta^\sharp \rangle \backslash M_d$ is a hyperbolic cone manifold so that 
the multicurve $c$ is the singular locus and the cone angle is at most $\pi$, we can 
now invoke a result of Kojima \cite{Ko98} and obtain the following.

\begin{thm}
The cone metric can be deformed to a complete finite volume hyperbolic 
metric on $M\setminus  c$.
\end{thm}

Together this yields the following result.

\begin{cor}\label{finitevol}
Let $c\subset M$ be any null homologous embedded finite collection of closed oriented geo\-de\-sics on 
the closed oriented hyperbolic 3-manifold $M$. Then 
$M\setminus c$ admits a complete hyperbolic metric of finite volume.
\end{cor}


An interesting question is how the hyperbolic metric on $M$ and the hyperbolic cone metric 
are related. A natural invariant to look at to this end is the volume of these metrics, which is 
well defined for hyperbolic cone metrics. The following is an easy consequence of 
a recent result of \cite{C++24}. 

\begin{prop}\label{volume}
The volume of the hyperbolic cone metric on $M$ is strictly larger than
the volume of the smooth hyperbolic metric. 
\end{prop}
\begin{proof}
As the cone angle of the cone metric $h$ on $M$ does not exceed $\pi$, 
the cone metric 
is a limit of smooth metrics of sectional curvature bounded 
from below by $-1$ and hence $M$ equipped with this cone metric is an 
${\rm RCD}(2,3)$-space. Thus the strict inequality 
${\rm vol}(M,h)> {\rm vol}(M,g)$ where $g$ is the smooth hyperbolic metric 
is an immediate consequence
of Corollary 1.5 of \cite{C++24}. 
\end{proof}
 
\begin{remark}
Proposition \ref{volume} can be reformulated as stating that the simplicial volume of 
the $d$-fold branched cover of $M$ is strictly larger than $d$ times
the simplicial volume of $M$. We refer to \cite{Fr17} for more information on this point of view.
\end{remark}

The volume of a hyperbolic 3-manifold can be studied using topological tools. To obtain an upper bound one can proceed
as follows. Choose an effective triangulation of a closed hyperbolic manifold $M$. This triangulation can be
\emph{straightened} to a triangulation with geodesic edges and totally geodesic facets. The hyperbolic 
volume of a straightened simplex does not exceed the hyperbolic volume $\kappa_n$ of an ideal hyperbolic 
tetrahedron, so the volume of $M$ is at most $\kappa_n \vert T\vert$ where $\vert T\vert$ denotes the number of 
maximal simplices of the triangulation. 

Lower bounds on the volume of a hyperbolic manifold in terms of topological information are much harder to obtain. The following was shown in  
\cite{HJ22}.

\begin{thm}
  For every $g\geq 2$ there exists a constant $C(g)>0$ so that the following holds true.
  The volume of a closed hyperbolic $3$-manifold of Heegaard genus $g$ is at least
  $c(g) d_{\rm Hempel}$ where $d_{\rm Hempel}$ denotes the Hempel distance of the Heegaard splitting.
\end{thm}  

Here a Heegaard splitting of genus $g\geq 1$ of a closed oriented 3-manifold $M$ is the decomposition of $M$ 
into two handlebodies of genus $g$. The manifold  $M$ is obtained from these handlebodies
by identifying their  boundaries with an orientation reversing diffeomorphism.
The Hempel distance is an invariant of the gluing diffeomorphism. We refer to
\cite{HJ22} for more information.



\section{Branched coverings in dimension $n\geq 4$}
\label{sec:branched4}

In this section 
we consider a degree $d$ cyclic covering $M_d\to M$ of a closed hyperbolic oriented manifold 
of dimension $n\geq 4$ 
branched along a closed totally geodesic oriented submanifold $\Sigma$ of codimension two.
We assume that there exists an isometric involution 
$\iota:M\to M$ whose fixed point set is a compact (possibly disconnected)
totally geodesic embedded hyperplane so that $\Sigma$ is contained in a component $H$ of this
fixed point set and is null-homologous in $H$. 
We refer to \cite{FP20,HJ24} for a discussion of examples. 
By Proposition \ref{prop:branched}, for every $d\geq 2$ there exists a cover $M_d$ 
of $M$ of degree $d$, branched
along $\Sigma$. Our goal is to prove the third 
part of Theorem \ref{mainthm} and show that for at most one $d\geq 2$, this cyclic branched
covering admits a hyperbolic metric. 

Our strategy is to assume that
$M_d$ admits a hyperbolic metric and seek to obtain a geometric understanding of this metric. 
We carry this out in three subsections.

\subsection{Fixed point sets of isometries}\label{fixedpointset}  

Since $\Sigma$ is homologous to zero in $H$, it 
bounds a submanifold $H_0\subset H$.
Put $H_1=H\setminus H_0$. 


Fix a number $d\geq 2$. 
The $d$-fold covering $X$ of $M$ branched along the totally geodesic submanifold 
$\Sigma \subset H\subset M$ can be realized as follows. 
Let $M_{\rm cut}$ be obtained from $M$ by cutting along $H_0$, that is, $M_{\rm cut}$ is the metric completion of
$M-H_0$. Thus $M_{\rm cut}$
is a compact (topological) 
manifold whose boundary consists of two copies $H_0^{-}$ and $H_0^{+}$ of $H_0$ intersecting in $\Sigma$. 
The manifold $X$ 
is obtained by gluing $d$ copies $M_{\rm cut}^{1},\dots,M_{\rm cut}^{d}$ of $M_{\rm cut}$ along the boundary, so that the copy of $H_0^+$ in $M_{\rm cut}^{i}$ is glued to the copy of $H_0^{-}$ in $M_{\rm cut}^{i+1}$ (where the superscripts $i$ are taken ${\rm mod}\, d$), see \cite{FS90} for a similar
construction. 
We denote by $H_0^{d,i}\subset X$ the copy of $H_0$ in $X$ which is the projection of the copy of 
$H_0^+$  in $M_{\rm cut}^i$.

Locally near $H$, the isometric involution $\iota=\iota_M$ 
acts as a reflection in $H$, exchanging the two components of $U\setminus H$ where
$U$ is a tubular neighborhood of $H$ in $M$. Thus $\iota_M$ acts as an involution 
$\iota_{\rm cut}:M_{\rm cut}\to M_{\rm cut}$ which 
exchanges $H_0^{+}$ and $H_0^{-}$ and fixes $W={\rm Fix}(\iota_M)\setminus H_0\supseteq H_1$.

As a consequence, $\iota_M$ induces an 
involution $\iota:X\to X$ with the property that $\iota(M_{\rm cut}^i)=M_{\rm cut}^{d+2-i}$ 
and so that the restrictions $\iota\vert M_{\rm cut}^{i}:
M_{\rm cut}^{i} \to M_{\rm cut}^{d+2-i}$ are identified with 
$\iota_{\rm cut}:M_{\rm cut} \to M_{\rm cut}$ (superscripts are again taken ${\rm mod}\, d$). 

If the degree $d$ of the covering is odd, then the fixed point set of this involution of $X$ 
is the union  of the copy $W^1$ of $W$ in $M_{\rm cut}^1$ 
and of $H_0^{d,(d-1)/2}$, glued to  
 $H_1^1\subset W^1$ along $\Sigma$. 
If $d$ is even then the fixed point set of $\iota$ consists of the two copies of $W$ 
in $M_{\rm cut}^1$ and $M_{\rm cut}^{d/2+1}$ glued along $\Sigma$. 

Let $\zeta$ be a generator of the cyclic deck group of $X\to M$. 
It cyclically permutes the 
copies $M^1_{\rm cut},\dots,M_{\rm cut}^d$ of $M_{\rm cut}$ in $X$. 
If the degree $d$ is even define
$j=\zeta \circ \iota$ (read from right to left), and for odd degree define $j=\iota$. 
The following is now immediate from the construction.

\begin{fact}\label{factfixed}
\begin{itemize}
    \item 
If $d$ is even, then the fixed point set of $j$ in $X$ is the union 
${\rm Fix}(j)=H_0^{d,1}\cup H_0^{d,1+d/2}$,  
and the copies $H_0^{d,1}$ and $H_0^{d,1+d/2}$ of $H_0$ are
glued along $\Sigma$.
\item If $d$ is odd then the fixed point set of $j$ in $X$ is ${\rm Fix}(j)=W^1\cup H_0^{d,(d-1)/2}$, and it is 
homeomorphic to ${\rm Fix}(\iota_M)$. 
\end{itemize}
\end{fact}


The fixed point set of each of the involutions $\zeta^i \circ j \circ \zeta^{-i}$ ($i=0,\dots,d-1$) 
is the embedded submanifold $\zeta^i({\rm Fix}(j))$ of $X$. Their union cuts $X$ up into the $d$ copies of $M_{\rm cut}$
if $d$ is even, and into $d$ copies of $M\setminus H$ if $d$ is odd. 
We call any diffeomorphism of $X$ contained in the finite group of diffeomorphisms
of $X$ generated by $j$ and $\zeta$ an \emph{admissible} diffeomorphism of $X$.

By Mostow rigidity, any homotopy self-equivalence $\sigma$ of $X$ is homotopic to a unique isometry $\sigma^\#$ of $X$.
Furthermore, by uniqueness, the map
\[
	{\rm Homeo}(X) \to {\rm Isom}(X), \, \sigma \mapsto \sigma^\#
\]
which associates to a homeomorphism the unique isometry homotopic to it 
is a group homomorphism. 
The following result was established in \cite{HJ24}.

\begin{prop}[Proposition 5.8 of \cite{HJ24}]\label{fixed}
Let $\phi$ be 
an admissible diffeomorphism of $X$ 
and let $\phi^\#$ be the 
isometry of $X$ homotopic to $\phi$. Then 
the fixed point set 
${\rm Fix}(\phi^\#)\subset X$ of 
$\phi^\#$ 
is (abstractly) diffeomorphic to ${\rm Fix}(\phi)$. 
Moreover, ${\rm Fix}(\phi^\#)$ and ${\rm Fix}(\phi)$ are freely homotopic inside $X$.
\end{prop}


From now on we always denote by $F$ the component of ${\rm Fix}(j)$ containing $\Sigma$ and by
$F^\#$ the homotopic component of ${\rm Fix}(j^\#)$ whose existence is guaranteed by Proposition 
\ref{fixed}. 
By \Cref{fixed} and Mostow rigidity for closed hyperbolic manifolds of 
dimension $n-1\geq 3$, there exists an isometry $\psi:F\to F^\#$ which maps 
$\Sigma$ to the fixed point set $\Sigma^\#$ of $\zeta^\#$. Furthermore, with a homotopy we may
identify $\Sigma$ and $\Sigma^\#$ in $X$. 
For each $i=0,\dots,d-1$, the map $(\zeta^\#)^i\circ \psi \circ \zeta^{-i}$ maps
$\zeta^{i}(F)$ isometrically onto $(\zeta^\#)^i(F^\#)$.

Although by \Cref{fixed}, the cyclic group generated by $\zeta^\#$ acts freely on $X\setminus \Sigma^\#$
and the manifold $F^\#$ is homotopic to $F$, this does not necessarily imply that 
$\zeta^\#(F^\#)\cap F^\#=\Sigma^\#$. The following lemma takes care of this issue.

\begin{lem}[Lemma 5.9 of \cite{HJ24}]\label{intersection}
\begin{enumerate}
    \item The differential 
    of $\zeta^\#$ acts on the normal bundle of $\Sigma^\#$ by a rotation 
    with angle $2\pi/d$. 
\item We have $F^\#\cap \zeta^\#(F^\#)=\Sigma^\#$.
\end{enumerate}
\end{lem}

Let $N$ be the compact hyperbolic manifold with totally geodesic boundary $\partial N$ which is obtained by
cutting $M$ open along $H$, that is, $N$ is the metric completion of $M-H$. 
If $H$ is non-separating,
then $N$ is connected,  otherwise $N$ has two connected components. The 
boundary $\partial N$ 
of $N$ is totally geodesic and consists of two copies of $H$ containing one copy of $\Sigma$ each.  
The main tool towards the third part of  Theorem \ref{mainthm} is the following result. 

\begin{prop}\label{prop:boundaryrigidity}
If the cyclic $d$-fold branched cover $X$ of $M$ admits a hyperbolic metric, then
there exists a hyperbolic cone manifold 
$N^\# $ 
satisfying the following properties:
\begin{enumerate}
    \item $N^\#$ is homotopy equivalent to $N$.
   \item The boundary $\partial N^\#$ of $N^\#$ is path isometric to $\partial N$.
    \item The singular locus of $N^\#$ consists of the two copies of $\Sigma$ in $\partial N^\#$. The 
    cone angle at each of these copies  equals $2\pi/d$.
\end{enumerate}
\end{prop}

\begin{rem}\label{kerckhoffstorm}
As the dimension of $N$ is at least four, it
follows from  a result of Kerckhoff and Storm \cite[Theorem 2.5]{KS12} 
that there is no continuous deformation of the compact hyperbolic manifold $N$ with totally geodesic boundary,
viewed as a convex cocompact hyperbolic manifold,  to 
the hyperbolic cone manifold $N^\#$ with the same fundamental group 
within the space of convex cocompact hyperbolic manifolds. 

This does not reduce Theorem \ref{mainthm} to Proposition 
\ref{prop:boundaryrigidity} as the result of Kerckhoff and Storm does not rule out that there are 
isolated faithful convex cocompact representations of $\pi_1(N)$ into ${\rm PO}(n,1)$ 
which can not be deformed to the representation defining $N$. 
\end{rem}

\subsection{The proof of Proposition \ref{prop:boundaryrigidity}}\label{even}
Throughout we assume that 
$X$ admits a hyperbolic metric. 
Denote as before by $F$ the component of the fixed point set of 
the involution $j$ defined in Fact \ref{factfixed} which contains $\Sigma$. 
The union $\cup_i \zeta^i(F)$ 
of the corresponding components of the 
fixed point sets of the involutions $\zeta^i \circ j \circ \zeta^{-i}$ cut $X$ up into $d$ copies of the manifold 
$M_{\rm cut}$ if $d$ is even, and into $d$ copies of the manifold 
$N$ 
from \Cref{prop:boundaryrigidity} if $d$ is odd.

Using these conventions, Proposition \ref{fixed} shows that 
the fixed point set of 
each of the isometric 
involutions $(\zeta^\#)^i \circ j^\# \circ (\zeta^\#)^{-i}$ homotopic  to 
$\zeta^i \circ j \circ \zeta^{-i}$
has a component 
$(\zeta^\#)^i(F^\#)$ which 
is a hyperplane isometric to the component $\zeta^i(F)$ of the fixed point set of 
$\zeta^i \circ j\circ \zeta^{-i}$ and contains $\Sigma$. 
We shall show that the union $\cup_i (\zeta^\#)^i(F^\#)$ of these submanifolds of $X$ 
cut $X$ into $d$ sectors homotopy equivalent to $M_{\rm cut}$ if $d$ is even, and homotopy equivalent
to $N$ if $d$ is odd. This yields Proposition \ref{prop:boundaryrigidity} if $d$ is odd. 
If $d$ is even then such a sector 
contains an embedded isometric copy of $H_1$ which intersects the boundary of 
$M_{\rm cut}^\#$ in $\Sigma^\#$. Cutting $M_{\rm cut}^\#$ along this copy of $H_1$ 
then yields a hyperbolic cone manifold with the properties stated in 
Proposition \ref{prop:boundaryrigidity}.


For each $i=0,\dots,d-1$, 
\Cref{fixed} yields an isometry $\phi_i:\zeta^i(F) \xrightarrow{\cong} (\zeta^\#)^i(F^\#)$ such that the maps
\[
	{\rm incl}_{\zeta^i(F)}: \zeta^i(F) \hookrightarrow X 
	\quad \text{and} \quad 
	{\rm incl}_{(\zeta^\#)^i(F^\#)} \circ \phi_i: \zeta^i(F) \xrightarrow{\cong} (\zeta^\#)^i(F^\#) \hookrightarrow X
\]
are homotopic. As 
$(\zeta^\#)^i(F^\#) \subseteq X$ is totally geodesic,
$\phi_i(\Sigma) \subseteq X$ is a closed totally geodesic submanifold of codimension two which is in the same free homotopy class as $\Sigma$. Since each free homotopy class can contain at most one totally geodesic representative, this implies that $\phi_0(\Sigma)=\dots=\phi_{d-1}(\Sigma)$. We define $\Sigma^\#\subset X$ to be this closed totally geodesic submanifold of codimenson two. So, by construction, $\Sigma^\# \subseteq (\zeta^\#)^i(F^\#)$ for all $i=0,\dots,d-1$.

After possibly changing the hyperbolic metric of $X$ with an isotopy, we may assume that for each connected component $\Sigma_0$ of $\Sigma$ we have $\Sigma_0 \cap \Sigma_0^\# \neq \emptyset$, where $\Sigma_0^\#=\phi_0(\Sigma)$. So, for each component, we can fix a basepoint $x_0 \in \Sigma_0 \cap \Sigma_0^\#$, 
and we may assume without loss of generality that $\phi_0(x_0)=x_0$. We call such a basepoint \emph{preferred}. 
Due to \Cref{fixed}, we may also assume that
\[
    \pi_1(\Sigma_0,x_0)=\pi_1(\Sigma_0^\#,x_0)
    \quad \text{and} \quad
    \pi_1(F,x_0)=\pi_1(F^\#,x_0).
\]
In the sequel, the fundamental
group $\pi_1(X,x_0)$ will always be represented with respect to a fixed choice $x_0$ of 
preferred basepoint. 
We can now prove \Cref{prop:boundaryrigidity}.

\begin{proof}[Proof of \Cref{prop:boundaryrigidity}] 
By construction, the subspace $F\cup \zeta(F)$ of $X$ separates $X$. If $d$ is even, then
by the definition of the map $j$, 
the complement $X-(F\cup \zeta(F))$ contains two connected components whose closures are homeomorphic 
to $M_{\rm cut}$. If $d$ is odd then it contains one connected component whose closure
is homeomorphic
to $M_{\rm cut}$. In both cases, let $Z$ be the closure of such a component. 
Its boundary consists of 
two copies of $H_0$ glued along $\Sigma$. 

By Lemma \ref{intersection}, there exists a corresponding component $M_{\rm cut}^\#$ of 
$X-(F^\#\cup \zeta^\#(F^\#))$. The boundary of its closure $Z^\#$ 
is connected and consists of two copies of $H_0$ meeting 
along $\Sigma$ with an angle $2\pi/d$. Identifying $\Sigma$ and $\Sigma^\#$ as before and choosing
a basepoint $x\in \Sigma$, we claim that $\pi_1(Z,x)=\pi_1(Z^\#,x)$. 

Namely, as $\pi_1((\zeta^\#)^i(F^\#),x_0)=\pi_1(\zeta^i(F),x_0)$ for all
$i=0,\cdots,d-1$, it holds that $\pi_1(\partial Z,x)=
\pi_1(\partial Z^\#,x)$. As $\partial Z$ is a separating hypersurface in 
$X$ homotopic to $\partial Z^\#$, by the theorem of Seifert-van Kampen, we know that 
\begin{equation*}\pi_1(X,x) =\pi_1(Z,x)*_{\pi_1(\partial Z,x)} \pi_1(X\setminus Z,x)
=\pi_1(Z^\#,x)
*_{\pi_1(\partial Z,x)}\pi_1(X\setminus Z^\#,x).\end{equation*}
It then follows from the normal form for amalgamated products \cite[p.186]{LS01}  
that $\pi_1(Z^\#,x)$ is isomorphic to either 
$\pi_1(Z,x)$ or to $\pi_1(X\setminus Z,x)$. 

If $d=2$ then $\pi_1(Z,x)$ is isomorphic to $\pi_1(X\setminus Z,x)$ and the claim is clear. 
If $d\geq 3$ then note that $\zeta_*=\zeta_*^\#$ maps 
$\pi_1(Z,x)$ to a proper subgroup of $\pi_1(X\setminus Z,x)$, and it maps 
$\pi_1(X\setminus Z,x)$
to a proper supergroup of $\pi_1(Z,x)$. Furthermore, it maps 
$\pi_1(Z^\#,x)$ to a proper subgroup of $\pi_1(X\setminus Z^\#,x)$ and it maps 
$\pi_1(X\setminus Z^\#,x)$ to a proper supergroup of 
$\pi_1(Z^\#,x)$. Thus we have $\pi_1(Z,x)=\pi_1(Z^\#,x)$ as claimed.

Note that if $d$ is odd, then the component $Z^\#$ contains a totally geodesic hypersurface 
isometric to $H_1$ which intersects the boundary of $Z^\#$ along $\Sigma^\#$. There exists
an isometric involution of $M_{\rm cut}^\#$ which exchanges the two copies of 
$H_0$ in its boundary and hence $H_1$ meets the boundary of $M_{\rm cut}^\#$ with an angle of 
$\pi/d$. 
We refer to Fact \ref{factfixed} for more information.
Cutting $Z^\#$ open along this hypersurface  then yields a cone manifold with the properties stated in 
\Cref{prop:boundaryrigidity}. 

Now assume that the covering degree $d$ is even. Note that the roles of the hypersurfaces
$H_0$ and $H_1$ can be exchanged and hence there exists a second involution $j_0$ of $X$ 
whose fixed point set is 
the union
$H_1^1\cup H_1^{1+d/2}$ of the copies of $H_1$ in $M_{\rm cut}^1$ and $M_{\rm cut}^{1+d/2}$, glued along $\Sigma$, as fixed point set.
Denote by 
$j_0^\#$ the isometric involutions of $X$ freely homotopic to $j_0$. 
As $j$ and $j_0$ can be exchanged in the arguments used so far, we may assume that 
$j_0$ is admissible. 
Thus 
${\rm Fix}(j_0^\#)$ is a separating totally geodesic hyperplane in $X$ which is isometric to two copies of $H_1$ glued along $\Sigma$. 
This hyperplane contains $\Sigma^\#$ as $\Sigma^\#$ is the unique submanifold of $X$ homotopic to $\Sigma$ which contains
each closed geodesic in the free homotopy class of an element of $\pi_1(\Sigma)$. 
Furthermore, $j_0^\#$ fixes $\Sigma^\#$ pointwise, and it 
exchanges the two components of $X\setminus{\rm Fix}(j_0^\#)$. 

By Mostow rigidity, we have
\[j_0^\#= j^\# \circ \zeta^\#.\] In particular, 
for $x\in {\rm Fix}((\zeta^\#)^{-1}\circ j^\#\circ \zeta^\#)$ it holds $j_0^\#(x)=\zeta^\#(x)$. As a consequence, $j_0^\#$ 
induces an isometric involution of $M_{\rm cut}^\#$ which 
exchanges the two components of 
$\partial M_{\rm cut}^\#\setminus\Sigma^\#$. As $j_0^\#$ acts as a reflection along a connected separating hyperplane in $X$
containing $\Sigma^\#$, the restriction of $j_0^\#$ acts 
 as a reflection 
along a totally geodesic embedded hyperplane $H_1^\#$ which intersects $\partial M_{\rm cut}^\#$ in $\Sigma^\#$. 
In particular, as $j_0^\#$ is an isometry, the hyperplane $H_1^\#$ meets a component of 
$\partial M_{\rm cut}^\#\setminus\Sigma^\#$
along $\Sigma^\#$ with an angle of $\pi/d$. Thus cutting $M_{\rm cut}^\#$ open along $H_1^\#$ yields a hyperbolic manifold
$N^\#$ with piecewise totally geodesic boundary with properties (ii) and (iii) of  Proposition \ref{prop:boundaryrigidity}. 

That $N^\#$ is homotopy equivalent to $M\setminus H=N$ follows from the fact that both are aspherical manifolds with boundary
and isomorphic fundamental groups. 
\end{proof}

\subsection{Proof of Theorem \ref{mainthm}}

We showed so far 
that the existence of a hyperbolic metric on the $d$-fold 
covering $X$ of $M$ branched along
$\Sigma$ gives rise to a convex cocompact hyperbolic manifold $N_{d}$ 
with two boundary components, each of which 
is path isometric to the hypersurface $H$. The manifold is singular along $\Sigma\subset H$, with cone angle 
(or bending angle) $\pi/d$, and it is 
homotopy equivalent to the hyperbolic manifold $N$ with totally geodesic boundary obtained by cutting 
$M$ open along $H$. Note that $N$ is connected if and only if the hypersurface 
$H$ is non-separating. 

Construct a new manifold $W$ by gluing $2 d$ copies $N_{d}^i$ $(i=1,\dots, 2d)$ 
of $N_{d}$ along the boundary as follows. Let $\partial N_{d}^\pm$ be the two distinct
boundary components of $N$, and let $(\partial N_{d}^i)^\pm$ be the corresponding 
boundary components of $N_{d}^i$. Each of  these components contains
a copy of $\Sigma$ which decomposes the component into two connected components
$(H_{0,d}^i)^\pm, (H_{1,d}^i)^\pm$. For each odd $i\leq 2d$ identify 
$(H_{0,d}^i)^\pm$ with 
$(H_{0,d}^{i+1})^\pm$, and for even $i\leq 2d$ identify $(H_{1,d}^i)^\pm$ with 
$(H_{1,d}^{i+1})^\pm$. As the cone angle of $\partial N_{d}^{\pm}$ along $\Sigma$ equals $\pi/d$, 
the hyperbolic metrics on the bordered manifolds $N_{d}^i$ induce a smooth hyperbolic metric on 
$W$. 

\begin{lem}\label{twosheet}
The manifold $W$ is a two-sheeted unbranched covering of $X$.
\end{lem}

In fact, if the hypersurface $H$ in $M$ is non-separating, then 
the same holds true for the component $F$ containing $\Sigma$ of the fixed point set of 
the involution $j$ of $X$ specified in Fact \ref{factfixed}. 
Then $W$ is the two-sheeted covering of $X$ so that 
the preimage of $F$ consists of two components which separate $W$. 
If $H\subset M$ is separating, then $N$ consists of two connected 
components, and $W$ consists of two copies of $X$. 

\begin{proof}
The case that  $H$ is separating is clear from the above remark. Thus assume that 
$H$ is non-separating. Then the same holds true for the hypersurface $F$.
Cut $X$ open along $F$. The resulting manifold $Q$ is connected and has two 
boundary components $\partial Q^-,\partial Q^+$ 
which are homeomorphic and path isometric to $H$. The 
metric is singular along the two copies of the totally geodesic submanifold 
$\Sigma^\#$ in the two boundary components of $Q$. 

Glue a second copy $\hat Q$ of $Q$ to $Q$ along the boundary in such a way that the 
boundary component $\partial \hat Q^-$ is glued to the boundary component
$\partial Q^+$, and the boundary component $\partial \hat Q^+$ is glued to the 
boundary component $\partial Q^-$. The resulting manifold is equipped with a smooth 
hyperbolic metric, and admits an obvious two sheeted unbranched covering onto $X$.
The lemma follows. 
\end{proof}

\begin{proof}[Proof of Theorem \ref{mainthm}]
We divide the proof into two claims.

\smallskip\noindent
{\bf Claim 1:} \emph{Among the branched coverings of $M$ of even degree $d\in 2\mathbb{N}$, 
at most one can be homeomorphic 
to a hyperbolic manifold.}

\smallskip
\emph{Proof of Claim 1:}
We argue by contradiction and we assume that there are distinct multiples 
$d_1\not=d_2 \in 2\bbN$ such that the cyclic $d_i$-fold branched cover $X^{(d_i)}$ of $M$ 
admits a smooth hyperbolic 
metric for $i=1,2$. Then, for each $i=1,2$, 
\Cref{prop:boundaryrigidity} (see the end of the proof for an explicit statement) 
implies that there exists a hyperbolic cone manifold $M_{\rm cut}^{2\pi/d_i}$ with totally geodesic boundary $\partial M_{\rm cut}^{2\pi/d_i}$ 
homeomorphic and path isometric to $\partial M_{\rm cut}$, with singular set isometric to $\Sigma$, cone angle $2\pi/d_i$ along $\Sigma$, and $\pi_1(M_{\rm cut}^{2\pi/d_i})=\pi_1(M_{\rm cut})$.


Note that $\frac{d_1}{2}\frac{2\pi}{d_1}+\frac{d_2}{2}\frac{2\pi}{d_2}=2\pi$.
Therefore, we can glue $d_1/2$ copies of $M_{\rm cut}^{2\pi/d_1}$ and 
$d_2/2$ copies of $M_{\rm cut}^{2\pi/d_2}$ in cyclic order along the components of $\partial M_{\rm cut}^{2\pi/d_i} \setminus \Sigma$ to a smooth hyperbolic manifold $Y$.
An application of the Seifert--van Kampen theorem shows that the fundamental group 
of $Y$ is isomorphic to the fundamental group of the $(d_1+d_2)/2$-fold cyclic cover $X$ of $M$ branched along $\Sigma$. 
In particular, this fundamental group admits a finite group of automorphisms generated 
by an element $\zeta_\ast$ of order $(d_1+d_2)/2$ and an involution $j_\ast$ corresponding to the 
automorphisms induced by the homeomorphisms $\zeta$ and $j$ of $X$ (notations are as before). By the proof of 
\Cref{prop:boundaryrigidity}, for each $i=0,\dots,(d_1+d_2)/2-1$, the fixed point group of $\zeta_\ast^i \circ j_\ast \circ \zeta_\ast^{-i}$ is the fundamental group of an embedded codimension one submanifold $F_i$ that, by construction of the hyperbolic metric on $Y$, is already totally geodesic.
Moreover, for some $i$ the totally geodesic submanifolds $F_i$ and $F_{i+1}$ intersect with angle $2\pi/d_1$, while for other $i$ they intersect with angle $2\pi/d_2$.

By Mostow rigidity, there exist isometries $\zeta^\#, j^\#$ of the hyperbolic manifold $Y$ of order $(d_1+d_2)/2$ and $2$, respectively, that induce the outer automorphism given by $\zeta_\ast$ and $j_\ast$. By \Cref{intersection}, the fixed point set of $\zeta^\#$ is a 
codimension two totally geodesic submanifold $\Sigma^\#$ freely homotopic to $\Sigma$, and thus $\Sigma^\#=\Sigma$ since $\Sigma$ is already totally geodesic in $Y$. 
Similarly, by \Cref{prop:boundaryrigidity}, the fixed point set $(\zeta^\#)^i(F^\#)$ of the involution 
$(\zeta^\#)^i \circ j^\# \circ (\zeta^\#)^{-i}$
is a totally geodesic hyperplane freely homotopic 
to the manifold $F_i$ satisfying
$\pi_1(F_i)={\rm Fix}(\zeta_\ast^i \circ j_\ast \circ \zeta_\ast^{-i})$, and thus $(\zeta^\#)^i(F^\#)=F_i$ since $F_i$ is already hyperbolic.
However, as $\zeta^\#$ acts by rotation with a fixed angle in the normal bundle
of $\Sigma$, the intersection angle of $(\zeta^\#)^i(F^\#)$ and $(\zeta^\#)^{i+1}(F^\#)$ is
the same for all $i$. But this contradicts the fact that, by construction,
the intersection angle of $F_i$ with $F_{i+1}$ varies between $2\pi/d_1$ and $2\pi/d_2$,
completing the proof of the claim.
\hfill$\blacksquare$

\smallskip\noindent
{\bf Claim 2:} \emph{No branched covering of $M$ of odd degree $d\geq 3$
can be homeomorphic 
to a hyperbolic manifold.}

\smallskip
\emph{Proof of Claim 2:}
Assume that there exists a covering $X$ of $M$ branched along $\Sigma$ of odd degree $d$ which admits a hyperbolic metric.
By Proposition \ref{prop:boundaryrigidity}, there exists a hyperbolic cone manifold 
$N_d$ homotopy equivalent to the manifold $N$ with cone angle $\pi/d$ along the copies of 
$\Sigma$ in each boundary component of $N_d$. Glue $d$ copies of $N_d$ to 
the manifold $N=M\setminus H$  along the boundary
as described in Lemma \ref{twosheet}. Note that this is possible because $d$ is odd. The resulting 
manifold is homotopy equivalent to a double unbranched covering of the branched covering  
$X$ of $M$ of degree $\frac{d+1}{2}$ as in Lemma \ref{twosheet}, 
and it is equipped with a smooth hyperbolic metric.

On the other hand, as $W\to X$ is a two-sheeted unbranched covering, the hyperbolic metric
on $X$ lifts to a hyperbolic metric on $W$. 
However, it follows precisely as in the proof of Claim 1 that this leads to a contradiction. 
\hfill$\blacksquare$
\end{proof}


\end{document}